\documentclass{article}
\usepackage[utf8]{inputenc}
\pdfoutput=1
\usepackage[english]{babel}
\usepackage{amsthm,amsmath,amssymb,soul}
\usepackage[all]{xy}
\usepackage{xy,mathrsfs,pstricks,color,verbatim,numprint}
\usepackage{color,transparent,subfigure,soul}  
\usepackage{tikz,ltablex,booktabs,tikz-cd}
\usepackage{tabstackengine}
\usetikzlibrary{arrows,decorations.pathmorphing}
\usetikzlibrary{positioning,arrows.meta,decorations.text,calc}
\usetikzlibrary{shapes.multipart}

\newcommand{\bsm}{\left[ \begin{smallmatrix}}
\newcommand{\esm}{\end{smallmatrix} \right]}
\newdir{ >}{{}*!/-5pt/\dir{>}}
\usepackage[a4paper]{geometry}
\usepackage[utf8]{inputenc}
    
\newcommand{\imono}[1]{ \;\xymatrix{  \ar@{>->}^{#1}[r] &  \\} }
\newcommand{\iepi}[1]{ \;\xymatrix{  \ar@{->>}^{#1}[r] &  \\} }
\newcommand{\mono}{ \;\xymatrix{  \ar@{>->}[r] &  \\} }
\newcommand{\epi}{ \xymatrix{   \ar@{->>}[r] &  \\} }

\newcommand{\A}{\mathcal{A}}
\newcommand{\E}{\mathcal{E}}

\newcommand{\Hom}{Hom}

\theoremstyle{definition} 
\newtheorem{theorem}{Theorem}[section]
\newtheorem{proposition}[theorem]{Proposition}
\newtheorem{lemma}[theorem]{Lemma}
\newtheorem{cor}[theorem]{Corollary}
\newtheorem{remark}[theorem]{Remark}
\newtheorem{example}[theorem]{Example}
\newtheorem{definition}[theorem]{Definition}

\theoremstyle{definition}

\def\Hom{\mbox{Hom}}

\def\Im{\mbox{Im}\, }
\def\Coim{\mbox{Coim}\,}
\def\Ker{\mbox{Ker}\,}
\def\Coker{\mbox{Coker}}

\title{Admissible intersection and sum property}

\author{Souheila Hassoun and Sunny Roy}
\begin{document}
\date{}
\maketitle
\abstract{
We introduce subclasses of exact categories in terms of
admissible intersections or admissible sums or both at the same time. These categories are recently studied by Br\"ustle, Hassoun, Shah, Tattar and Wegner to give a characterisation of quasi-abelian categories in \cite{HSW} and a characterisation of abelian categories in \cite{BHT}.
We also generalise the Schur lemma to the context of exact categories.}

\bigskip


\section{INTRODUCTION}

The Schur lemma is an elementary but extremely useful statement in representation theory of groups and algebras.\\
The lemma is named after Issai Schur who used it to prove
orthogonality relations and develop the basics of the representation theory of finite groups. Schur's lemma admits generalisations to Lie groups and Lie algebras, the most common of which is due to Jacques Dixmier.\\
The Schur Lemma  appears also in the study of stability conditions: when  an abelian category $\mathcal{A}$ is equipped with a stability condition, then  every endomorphism of a stable object is either the zero morphism or is an isomorphism, see \cite{Ru,BST}.
More generally, for $E_1, E_2 \in \mathcal{A}$ two stable objects of the same slope $\phi(E_1) =\phi(E_2)$, any morphism from $E_1$ to $E_2$ is either the zero morphism or is an isomorphism.
Bridgeland, in his seminal work on stability conditions on triangulated categories \cite{Br}, identifies the need to define a notion of stability on quasi-abelian categories, equipped with the exact structure of strict morphisms. 
This motivates the study of the Schur Lemma in the context of exact categories.\\
Exact categories generalise the abelian categories, namely additive categories with a choice of a Quillen exact structure \cite{Qu} which is given by a class of short exact sequences, called admissible pairs of morphisms, satisfying Quillen's axioms.\\
The notion of abelian category is an abstraction of basic properties of the category $Ab$ of abelian groups, more generally of the category
$Mod(R)$ of modules over some ring $R$. So it is not difficult to check that what holds for these categories generalise also to the abelian context.\\
In \cite{Bau}, Baumslag  gave a short proof of the Jordan–Hölder theorem for \emph{groups} by intersecting the terms in one subnormal series with those in the other series. The Schur lemma and classical isomorphism theorems for categories of modules play a crucial role in the proof.\\
Our motivation is to generalise Baumslag's idea, so we first generalise the Schur lemma to the context of exact categories and it turns out that the new version holds for any exact structure:
\begin{proposition}[Proposition \ref{schur}]({\bf{The $\E$-Schur lemma}}) 
 Let $\begin{tikzcd} X\arrow[r, "\circ" description, "f"] & Y \end{tikzcd}$
be an admissible non-zero morphism, that is, $f$ can be factored as an admissible epic followed by an admissible monic. Then, the following hold
\begin{enumerate}
\item[$\bullet$]if $X$ is $\E-$simple, then $f$ is an admissible monic,
\item[$\bullet$]if $Y$ is $\E-$simple, then $f$ is an admissible epic.
\end{enumerate}
\end{proposition}
Secondly, we study the notions of abelian intersections and sum, aiming at a generalisation. The abelian intersection, which exists and is well defined in a pre-abelian exact category, is not necessarily an \emph{admissible} subobject. So we introduce the following  exact categories which are quasi-n.i.c.e.
in the sense that they are {\bf n}ecessarily {\bf i}ntersection {\bf c}losed {\bf e}xact categories that do not necessarly admit admissible sums, and which we call {\bf{A.I}} since they admit {\bf{A}}dmissible {\bf{I}}ntersections:

\begin{definition}[\ref{quasi-nice}]
An exact category $(\A, 
\E)$ is called an \emph{AI-category} if $\A$ is a  pre-abelian additive category  satisfying the following additional axiom:
\begin{itemize}
\item[$({AI})$] 
 The pull-back $A$ of two admissible monics $j: C \rightarrowtail D$ and $g: B\rightarrowtail D$ exists and yields two admissible monics $i$ and $f$.
\[
\begin{tikzcd}
{A} \arrow[r, tail, "i"] \arrow[d, "f"', tail] & {B} \arrow[d, "g", tail] \\
{C} \arrow[r, tail, "j"']           & {D} \arrow[ul, phantom, "\lrcorner" very near start]          
\end{tikzcd}
\]

\end{itemize}

\end{definition} 
Inspired by the abelian sum, we also introduce exact categories satisfying the admissible sum property, that we call {\bf{A.S}} exact categories, since they admit {\bf{A}}dmissible  {\bf{S}}ums:
\begin{definition}[\ref{AS}]
An exact category $(\A, \E)$ is called an \emph{AS-category} if it satisfies the following additional axiom:
\begin{itemize}
\item[$({AS})$]The morphism $u$ in the diagram below, given by the universal property of the push-out $E$ of $i$ and $f$, is an admissible monic.
\[
\begin{tikzcd}
{A} \arrow[r, tail, "i"] \arrow[d, "f"', tail] & {B} \arrow[d, "l", tail]  \arrow[ddr, "g", tail, bend left]& \\
{C} \arrow[r, tail, "k"']   \arrow[drr, tail, "j"', bend right]        & {E} \arrow[ul, phantom, "\ulcorner" near end]  \arrow[dr, tail, "u"] \\ & & D       
\end{tikzcd}
\]
\end{itemize}
\end{definition}
Combining these two new notions, we introduce a special sub-class of the AI exact categories, that we call {\bf{A.I.S}} exact categories, since they admit {\bf{A}}dmissible {\bf{I}}ntersections and {\bf{S}}ums. These categories were called \emph{nice exact categories} in a previous version of this work:
\begin{definition}[\ref{nice}]
An exact category $(\A, 
\E)$ is an AIS-category or a \emph{nice} or if it satisfies the (AIS) axiom which is defined by both the (AI) and the (AS) axioms at the same time.
\end{definition}
These categories were recently studied by the first author, Thomas Br\"ustle, Amit Shah, Aran Tattar and Sven-Ake Wegner and we obtained the following characterisations:
\begin{theorem}\cite[Theorem 6.1]{HSW}
 A category $(\mathcal{A}, \E_{max})$ is quasi-abelian if and only if it is an AI-category.
\end{theorem} 

\begin{theorem}\cite[Theorem 4.22]{BHT}
An exact category $(\A, \E)$ is an AIS-category if and only if $\A$ is abelian and $\E=\E_{all}$.
\end{theorem}
These results make us conclude that the pull-back and push-out notions of unique intersection and sum do not always apply to all exact categories. This motivates the study of general admissible intersection and sum in \cite[Section 5]{BHT}, where we introduce (\cite[Definition 5.5]{BHT}) a notion of intersection and sum that works for all exact categories, and using this we study the Jordan-H\"older exact categories \cite[Theorem 5.11, 6.8, 6.13]{BHT}.\\

Finally, we reprove the classical isomorphism theorems from module theory using exact categorical arguments and we apply it all in the last section, where we fix an abelian category $\A$ with its maximal exact structure $\E_{all}$ given by the class of all short exact sequences and follow Baumslag's ideas to obtain a proof of the Jordan-H\"older theorem for abelian categories using the language of exact structures.
This proof is different than the abelian proof studied in \cite[Section 4.5, page 174]{par}.\\
Note that, parallel to our work, Enomoto studies the Schur lemma in \cite{E20} from the viewpoint of semibricks and wide subcategories.


\paragraph{Acknowledgements.}
The authors would like to thank their supervisor Thomas Br\"ustle
for his support, and would like also to thank Aran Tattar, Amit Shah, Sven-Ake Wegner and Haruhisa Enomoto for interesting discussions.\\
The authors were supported by Bishop's University, Université de Sherbrooke, and NSERC of Canada.
The first author is supported by the scholarship "thésards étoiles" of the ISM.

\section{Background}


In this section we recall from \cite{GR,Bu} the definition of Quillen exact structures and the definition of a pre-abelian additive category.

\begin{definition}{}Let $\mathcal{A}$ be an additive category. A kernel-cokernel pair $(i, d)$ in $\mathcal{A}$ is a pair of composable morphims such that $i$ is kernel of $d$ and $d$ is cokernel of $i$.
If a class $\mathcal{E}$ of kernel-cokernel pairs on $\mathcal{A}$ is fixed, an {\em admissible monic} is a morphism $i$ for which there exist a morphism $d$ such that $(i,d) \in \mathcal{E}$. An {\em admissible epic} is defined dually. Note that admissible monics and admissible epics are referred to as inflation and deflation in \cite{GR}, respectively. 
We depict an admissible monic by  $\mono$
and an admissible epic by $\epi$.
An {\em exact structure} $\mathcal{E}$ on $\A$ is a class of kernel-cokernel pairs $(i, d)$ in $\A$ which is closed under isomorphisms and satisfies the following axioms:
\begin{enumerate}
\item[(A0)] For all objets  $A \in Obj\mathcal{A}$ the identity $1_A$ is an admissible monic
\item[{(A0)$^{op}$}] For all objets  $A \in Obj\mathcal{A}$ the identity $1_A$ is an admissible epic

\item[(A1)] the class of admissible monics is closed under composition
\item[{(A1)}$^{op}$] the class of admissible epics is closed under composition
\item[(A2)] 
 The push-out of an admissible monic $i: A \mono B$ along an arbitrary morphism $f: A \to C$ exists and yields an admissible monic $j$:
\[\xymatrix{
A \; \ar[d]_{f} 
\ar@{ >->}[r]^{i}  \ar@{}[dr]|{\text{PO}} 
& B\ar[d]^{g}\\
C \; \ar@{>->}[r]^{j} & D}
\]

\item[{(A2)}$^{op}$]The pull-back of an admissible epic $h$ along an arbitrary morphism $g$ exists and yields an admissible epic $k$
\[\xymatrix{
A \; \ar[d]^{f} 
\ar@{ ->>}[r]^{k} \ar@{}[dr]|{\text{PB}}  & B\ar[d]^{g}\\
C \; \ar@{->>}[r]^{h} & D}
\]
\end{enumerate}
An {\em exact category} is a pair $(\mathcal{A}, \mathcal{E})$ consisting of an additive category $\mathcal{A}$ and an exact structure $\mathcal{E}$ on $\mathcal{A}$. The pairs $(i,d)$ forming the class $\mathcal{E}$ are called {\em admissible short exact sequences}, or just {\em admissible sequences.}

\end{definition}

\begin{definition}
\cite[Definition 8.1]{Bu}\label{ad mor}
A morphism $f: A \rightarrow B$ in an exact category is called \emph{admissible} if it factors as a composition of an admissible monic with an admissible epic. Admissible morphisms will sometimes be displayed as  
\[ \begin{tikzcd} A\arrow[r, "\circ" description, "f"] & B \end{tikzcd} \]
in diagrams, and the classes of admissible arrows of $\A$ will be denoted as ${\Hom^{ad}_{\A}}(-,-)$.
\end{definition}

\begin{proposition}
\cite[Proposition 2.16]{Bu}\label{obscure axiom}
Suppose that $i: A\rightarrow B$ is a morphism in $\A$ admitting a cokernel. If there exists a morphism $j:B\rightarrow C$ such that the composite $j\circ i: A\mono C$ si an admissible monic, then $i$ si an admissible monic.
\end{proposition}

\begin{definition}An additive category $\A$ is  \emph{pre-abelian} if it has kernels and cokernels.
\end{definition}

\begin{example}An additive category $\A$ is \emph{abelian} if it is \emph{pre-abelian} and all morphisms are \emph{strict}. So abelian categories are an example of pre-abelian additive categories where every morphism is admissible.

\end{example}


\section{The $\E$-Schur lemma}
In this section we generalise the abelian Shur lemma to the context of exact categories.

\begin{definition}\cite[Definition 3.1]{BHLR} 
Let $A$ and $B$ be objects of an exact category $(\mathcal{A},\E)$. If there is an admissible monic $i: A \rightarrowtail B$ we say the pair $(A,i)$ is an {\em admissible subobject} or {\em $\mathcal{E}-$subobject of $B$}. Often we will refer to the pair $(A,i)$ by the object $A$ and write $A {\subset}_{\mathcal{E}} B $.
If $i$ is not an isomorphism, we use the notation  $A {\subsetneq}_{\mathcal{E}} B $ and if, in addition, $A \not \cong 0$ we  say that $(A,i)$ is a \emph{proper} admissible subobject of $B$.
\end{definition}

\begin{definition}\cite[Definition 3.3]{BHLR}
A non-zero object $S $ in $(\A,\E)$ is {\em $\mathcal{E}-$simple} if $S$ admits no $\E-$sub\-objects except $0$ and $S$, that is, whenever $ A \subset_\E S$, then $A$ is the zero object or isomorphic to  $S$.
\end{definition}

\begin{remark}\label{quotient}
Let $A$ be an $\mathcal{E}-$subobject of $B$ given by the monic $A{\imono{i}} B$.
We denote by $B{/}^{i}A$ (or simply $B/A$ when $i$ is clear from the context) the Cokernel of $i$, thus we denote the corresponding admissible sequence as
\[ A \imono{i} B\epi B/A.\]
\end{remark}

\begin{remark}\label{zero coker}
An admissible monic $A \imono{i} B$ is relatively proper precisely when its co\-kernel is non-zero. 
In fact, by uniqueness of kernels and cokernels, the exact sequence 
$$B\imono{1_B} B\epi 0$$ 
is, up to isomorphism, the only one with zero cokernel. Thus an admissible monic $i$ has $\Coker(i) = 0$ precisely when $i$ is an isomorphism. Dually, an admissible epic $B \iepi{d} C$ is an isomorphism precisely when $\Ker(d) = 0$. In particular a morphism
which is at the same time an admissible monic and epic is an isomorphism.\\
Note that a subobject is proper means {\em all} admissible monics are proper.
\end{remark}

\begin{lemma}({\bf{The $\E$-Schur lemma}})\label{schur}
Let $\begin{tikzcd} X\arrow[r, "\circ" description, "f"] & Y \end{tikzcd}$
be an admissible non-zero morphism.
\begin{enumerate}
\item[$\bullet$]if $X$ is $\E-$simple, then $f$ is an admissible monic,
\item[$\bullet$]if $Y$ is $\E-$simple, then $f$ is an admissible epic.
\end{enumerate}
\end{lemma}
\begin{proof}
Let 
\[ \xymatrix{X\ar[rd]^{f} \ar@{->>}[d]_{e}  \\ S \;\ar@{>->}[r]^-m  & Y } \]

be the factorisation of $f$ as a composition of an admissible epic $e$  with an admissible monic $m$.
\begin{enumerate}
\item[$\bullet$]
if $X$ is $\E-$simple then either $\Ker(e)=X$ or $\Ker(e)=0$, but in the first case $e=0$ and so $f=0$,  contradicting the assumption $f\neq 0$. Hence $Ker(e)=0$, and by Remark \ref{zero coker}, $e\cong 1_{X}$ and $f \cong m$ and therefore $f$ is an admissible monic.

\item[$\bullet$] If $Y$ is $\E-$simple, then the $\E-$subobject $S$ is either zero or equal to $Y$, but in case $S=0$, $e=0$, we get $f=m\circ e=0$ which contradicts $f\neq 0$. Therefore $S=Y$ and $m: Y\mono Y$ is an admissible monic with zero cokernel. By  Remark \ref{zero coker},  $m\cong 1_{Y}$, and $f\cong e$, which means that $f$ is an admissible epic.

\end{enumerate}
\end{proof}

\begin{cor}\label{Aut}
 Let $S$ be an $\E-$simple object, then the non-zero admissible endomorphisms $\begin{tikzcd} S\arrow[r, "\circ" description, "f"] & S \end{tikzcd}$ form the group Aut$(S)$ of automorphisms of $S$.
\end{cor}
\begin{proof}
It follows from Lemma \ref{schur} that any non-zero admissible morphism $\begin{tikzcd} S\arrow[r, "\circ" description, "f"] & S \end{tikzcd}$ is an admissible monic and an admissible epic, thus $f$ is an isomorphism.\\
Conversely, every isomorphism is admissible, so we get the group of automorphisms of $S$ which is closed under composition by (A2) or $(A2)^{op}$.
\end{proof}
\begin{remark}
The classical Schur lemma on abelian categories states that the endomorphism ring of a simple object is a division ring. We show in Corollary \ref{Aut} that any non-zero admissible endomorphism of an $\E-$simple object is invertible, but it is not true in general that the set of admissible endomorphisms forms a ring.
 In fact, the composition of admissible morphisms need not be admissible, (see \cite[Remark 8.3]{Bu}), nor is it true for sums of admissible morphisms, as we discuss in \cite{BHT}. 
\end{remark}

\section{AI, AS and AIS-CATEGORIES }
Let us first recall the definitions of \emph{intersection} and \emph{sum} of subobjects for abelian categories in general as mentioned in \cite[section 5]{Gabriel}
 or as defined in \cite[Definition 2.6 ]{Po}:
 \bigskip
 
\begin{definition}\label{ab sum}Let $(X_1,i_1)$, $(X_2,i_2)$ be two subobjects of an object $X$ in an abelian category, that is, we consider monics $i_1:X_1\to X$ and $i_2:X_2\to X$. We denote by $X_1{+}_{X}X_2$ (or simply $X_1+X_2$ when there is no possibility of confusion) the {\em sum of $X_1$ and $X_2$}, which is defined as the image $\Im(s)$ of the morphism \[s=[i_1 \; i_2]: X_1\oplus X_2\rightarrow X.\] 

\bigskip
\end{definition}

\begin{definition}\label{ab intersection}
Let $(X_1,i_1)$, $(X_2,i_2)$ be two subobjects of an object $X$ in an abelian category. We denote by $X_1{\cap}_X X_2$ (or simply $X_1{\cap} X_2$)
the {\em intersection of $X_1$ and $X_2$}, defined as the kernel $\Ker(t) $ of the  morphism 
\[t = \begin{bmatrix}
d_1 \\
d_2 \\
\end{bmatrix}: X\rightarrow Y_1\oplus Y_2\] 
where $d_1:X \to Y_1$ and $d_2: X \to Y_2$ are the cokernels of the monics $i_1$ and $i_2$, respectively. 
\end{definition}
\bigskip

Note that this intersection, which exists and is well defined in a pre-abelian exact category, is not necessarily an \emph{admissible} subobject. So let us introduce the following  exact categories which are quasi-n.i.c.e.
in the sense that they are {\bf n}ecessarily {\bf i}ntersection {\bf c}losed {\bf e}xact categories that does not necessarly admit admissible sums, and which we call {\bf{A.I}} since they admit {\bf{A}}dmissible {\bf{I}}ntersections:

\begin{definition}\label{quasi-nice}
An exact category $(\A, 
\E)$ is called an \emph{AI-category} if $\A$ is  pre-abelian additive category  satisfying the following additional axiom:
\begin{itemize}
\item[$({AI})$] 
 The pull-back $A$ of two admissible monics $j: C \rightarrowtail D$ and $g: B\rightarrowtail D$ exists and yields two admissible monics $i$ and $f$.
\[
\begin{tikzcd}
{A} \arrow[r, tail, "i"] \arrow[d, "f"', tail] & {B} \arrow[d, "g", tail] \\
{C} \arrow[r, tail, "j"']           & {D} \arrow[ul, phantom, "\lrcorner" very near start]          
\end{tikzcd}
\]

\end{itemize}

\end{definition} \label{nice def}
Let us aslo introduce exact categories satisfying the admissible sum propertys, that we call {\bf{A.S}} exact categories, since they admit {\bf{A}}dmissible  {\bf{S}}ums:
\begin{definition}\label{AS}
An exact category $(\A, \E)$ is called an \emph{AS-category} if it satisfies the following additional axiom:
\begin{itemize}
\item[$({AS})$]The morphism $u$ in the diagram below, given by the universal property of the push-out $E$ of $i$ and $f$, is an admissible monic.
\[
\begin{tikzcd}
{A} \arrow[r, tail, "i"] \arrow[d, "f"', tail] & {B} \arrow[d, "l", tail]  \arrow[ddr, "g", tail, bend left]& \\
{C} \arrow[r, tail, "k"']   \arrow[drr, tail, "j"', bend right]        & {E} \arrow[ul, phantom, "\ulcorner" near end]  \arrow[dr, tail, "u"] \\ & & D       
\end{tikzcd}
\]
\end{itemize}
\end{definition}
Let us now introduce a special sub-class of the AI exact categories, that we call {\bf{A.I.S}} exact categories, since they admit {\bf{A}}dmissible {\bf{I}}ntersections and {\bf{S}}ums. These categories were called \emph{nice exact categories} in a previous version of this work:
\begin{definition}\label{nice}
An exact category $(\A, 
\E)$ is an AIS-category or a \emph{nice} or if it satisfies the following additional axiom:
\begin{itemize}
\item[$(AIS)$] 
 The pull-back of two admissible monics $j: C \mono D$ and $g: B\mono D$ exists and yields two admissible monics $i$ and $f$:
\[\xymatrix{
A \; \ar@{>->}[d]_{f} 
\ar@{ >->}[r]^{i}  \ar@{}[dr]|{\text{PB}} 
& B\ar@{>->}[d]^{g}\\
C \; \ar@{>->}[r]^{j} & D}
\]

and moreover, the push-out along these pull-backs yields an admissible monic $u$\footnote{The existence of $u$ is given by the universal property of the push-out.}:

\[\xymatrix{
PB \; \ar@{ >->}[d]^{f} 
\ar@{ >->}[r]^{i}  \ar@{}[dr]|{\text{PO}} 
& B\ar@{ >->}[d]^{l} \ar@{ >->}[ddr]^{g} \\
C \; \ar@{ >->}[drr]^{j} \ar@{>->}[r]^{k} & PO \ar@{ >->}[dr]^{u}
\\ & & D.}
\]
\end{itemize}

\end{definition}

 Now let us define relative notions of intersection and sum:
\begin{definition} \label{intersection & sum}

Let $(X_1,i_1)$, $(X_2,i_2)$ be two $\E$-subobjects of an object $X$. We define their \emph{intersection} $X_1{\cap}_X X_2$, to be the pullback
\[ \begin{tikzcd} X_1{\cap}_X X_2  \arrow[dr, phantom, "\lrcorner", very near end]  \arrow[r, "s_1", tail] \arrow[d, "s_2"', tail] & X_1 \arrow[d, "i_1", tail] \\ X_2 \arrow[r, "i_2"', tail] & X. \end{tikzcd} \] 
We then define their \emph{sum}, $X_1{+}_{X}X_2$, to be the pushout  
\[ \begin{tikzcd} X_1{\cap}_X X_2    \arrow[r, "s_1", tail] \arrow[d, "s_2"', tail] & X_1 \arrow[d, "j_1", tail] \\ X_2 \arrow[r, "j_2"', tail] & \arrow[ul, phantom, "\ulcorner", near end]X_1{+}_{X}X_2. \end{tikzcd}  \]
\end{definition}
\begin{remark}
Assume $(\A,\E)$ is an AIS-category, then these intersection and sum are well-defined and admissible for any two admissible subobjects.
\end{remark}

\begin{remark}   Let  $(X_1,i_1)$, $(X_2,i_2)$ and $(Y,j)$ be $\E$-subobjects of an object $X$. Then
\begin{enumerate}
    \item[i)]$X_1 \cap_X X_1= X_1 = X_1 +_X X_1$.
    \item[ii)]If $X_1{+}_{X}X_2=0_{\A}$ then $X_1=X_2=0_{\A}$.
     
     \end{enumerate}
\end{remark}

\begin{remark} \label{sum intersection as coker ker} Equivalently, for two $\E$-subobjects $(X_1,i_1)$, $(X_2,i_2)$  of an object $X$ we have 
\[
    X_1{\cap}_X X_2  = \Ker \left( \begin{tikzcd}[sep= large]
X_1 \oplus X_2 \arrow[r, "{[ i_1 - i_2] }"] & X \end{tikzcd} \right) \]
and 
\[
    X_1{+}_{X}X_2  = \Coker \left( \begin{tikzcd}[sep=large]
X_1 \cap_X X_2 \arrow[r, "{\left[ s_1  -s_2  \right]^t}"] & X_1 \oplus X_2
\end{tikzcd} \right). \]
Thus, as the direct sum is an associative operation, so are the sum and intersection operations. Moreover, the direct sum is commutative up to isomorphism, and so are the sum and intersection.
\end{remark}

Now let us show how this definition generalises the abelian versions from Definitions \ref{ab sum} and \ref{ab intersection}:
\begin{proposition}\label{intersection in abelian}

Let $(\A,\E_{all})$ be an abelian exact category and let $(X_1,i_1)$ and $(X_2,i_2)$ be two $\E$-subobjects of an object $X$. Then $\Ker t$ 
forms the pull-back of $(X, i_1, i_2)$, where

\[t = \bsm
d_1 \\
d_2 
\esm: X\rightarrow X/X_1\oplus X/X_2\]  is given by the cokernels $d_1,d_2$ of $i_1, i_2$ as in Definition \ref {ab intersection}.
\end{proposition}

\begin{proof}
Let us consider the following diagram

\[
\begin{tikzcd}
{\Ker t} \arrow[r, "k_1"] \arrow[d, "k_2"'] \arrow[rd, "i"] & {X_1} \arrow[d, tail, "i_1"]                &              \\
{X_2} \arrow[r, tail, "i_2"']                & {X} \arrow[r, two heads, "d_2"] \arrow[d, two heads, "d_1"'] \arrow[rd, "t"] & {X / X_2} \arrow[d, tail] \\
                                  & {X/ X_1} \arrow[r, tail]                      & {X/ X_1 \oplus X/ X_2}          
\end{tikzcd}
\]

where $\ i_1\circ k_1= i$ and $\ i_2\circ k_2= i$.

Assume now one has an object $V$ and two morphisms  $v_1$, $v_2$ such that 
$ i_1\circ v_1 = i_2\circ v_2$:

\[ 
\begin{tikzcd}
V \arrow[drr, "v_1", bend left] \arrow[ddr, "v_2"', bend right] \arrow[dr, dashed, "v"] & & & \\
&{\Ker t} \arrow[r, "k_1"] \arrow[d, "k_2"'] \arrow[rd, "i"] & {X_1} \arrow[d, tail, "i_1"]                &              \\ &
{X_2} \arrow[r, tail, "i_2"']                & {X} \arrow[r, two heads, "d_2"] \arrow[d, two heads, "d_1"'] \arrow[rd, "t"] & {X / X_2} \arrow[d, tail] \\ &
                                  & {X/ X_1} \arrow[r, tail]                      & {X/ X_1 \oplus X/ X_2}          
\end{tikzcd}
\]

Since $t\circ i_1\circ v_1 = \bsm
d_1 \\
d_2 \\
\esm \circ i_1\circ v_1 = \bsm
d_1\circ i_1\circ v_1 \\
d_2\circ i_1\circ v_1 \\
\esm = \bsm
0 \\
d_2\circ i_2\circ v_2 \\
\esm = 0 $,
by the universal property of the kernel there exists a unique morphism $v$ such that 
$i_1 \circ v_1= i\circ v = i_1 \circ k_2 \circ v$. Since $i_1$ is mono, we conclude
$v_1 = k_2\circ v$.

By symmetry we also have that there exists a unique morphism $v$ such that $v_2 = k_1\circ v$.

We conclude that $(\Ker t, k_1, k_2) $ is the pull-back of $(X, i_1, i_2)$.
\end{proof}

\begin{proposition}\label{sum in abelian}
Let $(\A, \E_{all})$ be an abelian exact category and let $(X_1,i_1)$ and $(X_2,i_2)$ be two $\E$-subobjects of an object $X$. Then $\Im s$ forms the push-out of
$(X_1{\cap}_X X_2, s_1, s_2)$ where $s$ is as in Definition \ref{ab sum} and $s_1$ and $s_2$ are given by the pull-back as in Definition \ref{intersection & sum}.
\end{proposition}

\begin{proof}
In the abelian case, the pull-back along $(X, i_1,i_2)$ is the kernel of $[i_1   \;\; i_2]: $
\[ \begin{tikzcd}[ampersand replacement = \&]  \Ker \bsm i_1 & i_2 \esm \arrow[r, "{\bsm s_1 \\ -s_2\esm}"] \& X_1 \oplus X_2 \arrow[r, "{\bsm i_1 & i_2 \esm}"] \arrow[dr, "{\bsm j_1 & j_2\esm}"'] \& X \\ \& \& \Coker \bsm s_1 \\ -s_2 \esm \end{tikzcd} \]
Consider the pull-back diagram defining $(X_1{\cap}_{X} X_2)= \Ker \bsm i_1 & i_2 \esm $
\[ \begin{tikzcd} X_1{\cap}_X X_2   \arrow[dr, phantom, "\lrcorner", near end]  \arrow[r, "s_1", tail] \arrow[d, "s_2"', tail] & X_1 \arrow[d, "i_1", tail] \\ X_2 \arrow[r, "i_2"', tail] & X \end{tikzcd} \]

and the push-out along $( s_1, s_2)$ :

\[ \begin{tikzcd}[row sep = large] X_1{\cap}_X X_2    \arrow[r, "s_1", tail] \arrow[d, "s_2"', tail] & X_1 \arrow[d, "j_1", tail] \\ X_2 \arrow[r, "j_2"', tail] & \arrow[ul, phantom, "\ulcorner", near end] X_1{+}_{X}X_2 = \Coker \bsm s_1 \\ -s_2 \esm . \end{tikzcd}  \]

The push-out $X_1{+}_{X}X_2$ is $\Coker(\Ker(s))=\mbox{Coim}\,(s).$
And since $\Coim(s)\cong \Im(s)$ in an abelian category, we conclude that $\Im(s)$ coincides with the general admissible sum in a nice category.
\end{proof}

\begin{cor}\label{abelian is nice}
Let $\A$ be an abelian category. Then $(\A,\E_{all})$ is an AIS-category.
\end{cor}
\begin{proof}
This follows directly from Propositions \ref{intersection in abelian} and \ref{sum in abelian}.
\end{proof}
\bigskip

Now we give some properties of the intersection and the sum of $\E-$subobjects of an object:
\begin{lemma}\label{intersection inclusion}
Let $X,Y$ and $Y'$ be $\E-$subobjects of an object $Z$ in an AI-category. If there exists an admissible monic \[ i : Y\mono Y'\] then there exists an admissible monic \[X{\cap}_{Z} Y\mono X{\cap}_{Z} Y'.\]
\end{lemma}

\begin{proof}By definition we have the two following pull-back diagrams

\[ \xymatrix{ X{\cap}_{Z} Y\ar@{>->}[r]^{f} \ar@{>->}[d]_{g} &  Y\ar@{>->}[d]_{h}\\ X  \ar@{>->}[r]_{k} & Z} \]

and

\[ \xymatrix{ X{\cap}_{Z} Y'\ar@{>->}[r]^{f'} \ar@{>->}[d]_{g'} &  Y'\ar@{>->}[d]_{h'}\\ X  \ar@{>->}[r]_{k} & Z} \]

where $\ h'\circ i= h$.

So we have a monic $\ i\circ f= l$ that commutes the following diagram

\[ \xymatrix{ X{\cap}_{Z} Y\ar@{>->}[r]^{l} \ar@{>->}[d]_{g} &  Y'\ar@{>->}[d]_{h'}\\ X  \ar@{>->}[r]_{k} & Z} \]

By the universal property of the pull-back, there exist a morphism 
 \[ r : X{\cap}_{Z} Y\to X{\cap}_{Z} Y'.\]

such that $\ f'\circ r= l$ and $\ g'\circ r= g$.

Since $l$ is an admissible monic, and the cokernel of $r$ exists, then the obscure axiom \ref{obscure axiom} implies that the morphism $r$ is also an admissible monic.

\end{proof}

\begin{lemma}\label{sum inclusion}
Let $X,Y$ and $Y'$ be $\E-$subobjects of an object $Z$ in an AS-category. If there exists an admissible monic \[ i : Y\mono  Y'\] then there exists an admissible monic \[Y{+}_{Z}X\mono Y'{+}_{Z}X\]
when these sums exists.
\end{lemma}

\begin{proof}

By definition we have the two following push-out diagram

\[\xymatrix{
X{\cap}_Z Y \; \ar@{ >->}[d]^{g} 
\ar@{ >->}[r]^{f}  \ar@{}[dr]|{\text{PO}} 
& Y\ar@{ >->}[d]^{d} \ar@{ >->}[ddr]^{l'} \\
X\; \ar@{ >->}[drr]^{e'} \ar@{>->}[r]^{e} & Y{+}_Z X \ar@{ >->}[dr]^{r'}
\\ & &Y'{+}_Z X  }
\]
where $\ d'\circ i= l'$, and by the universal property of the push-out, there exists a unique morphism 
 \[r' : Y{+}_{Z}X\to Y'{+}_{Z}X.\]
 such that $r'\circ e=e'$ and $r'\circ d=d'$.
The unique two admissible monics \[u : Y{+}_{Z}X\mono Z\]
 \[u' : Y'{+}_{Z}X\mono Z\]
such that $\ u'\circ r'= u$ are admissibles by the (AS) axiom, and since $u$ is an admissible monic and the cokernel of $r'$ exists, then the obscure axiom \ref{obscure axiom} implies that the morphism $r'$ is also an admissible monic. 
\end{proof}

\begin{proposition}Let $(X_1,i_1)$, $(X_2,i_2)$ be two $\E$-subobjects of an object $X$ in an AIS-category, then
$X_1 {\cap}_{X}X_2 = X_1 {\cap}_{(X_1{+}_{X}X_2)}X_2$.
\end{proposition}
\begin{proof}
Using the equivalent assertions of \cite[Proposition 2.12]{Bu}.
\end{proof}

\begin{definition}
An additive functor $F:\mathcal{A} \to \mathcal{B}$ is called \emph{exact} if for every kernel-cokernel pair $(i,d)$ in $\A$, we have that $(Fi, Fd)$ is a kernel-cokernel pair in $\mathcal{B}$. An additive functor $F:(\mathcal{A}, \E) \to (\mathcal{B}, {\E}')$ is called \emph{ $\E$-exact} if $F(\E) \subseteq {\E}'$.
\end{definition}

\begin{remark}
In particular, exact functors preserve kernels and cokernels and therefore preserve intersections and sums. 
\end{remark}

\section{ISOMORPHISM THEOREMS}
In this section $(\A,\E)$ is an AIS-category.\\
We will recall the existence of some special \emph{admissible} short exact sequences, which will play an important role in the proof of the Jordan-H\"older property.
\begin{lemma}\label{c3}Let $X$, and $Y'\mono Y''$ be three $\E-$subobjects of an object $Z$.
Then there exists an admissible short exact sequence
\[(Y'{+}_{Z}X)/X\mono (Y''{+}_{Z}X)/X\epi (Y''{+}_Z X)/(Y'{+}_Z X)\]
\end{lemma}
\begin{proof}

The admissible monic that exists by \ref{sum inclusion} fit into the commutative diagram below, where the arrow on the right exists by the universal property of a Cokernel, then by the dual of \cite[Proposition 2.12]{Bu}
the right square is bicartesian, and by (A2) (or by \cite[Proposition 2.15]{Bu}) the morphism \[Y'{+}_Z X/ X \imono{c} Y''{+}_Z X/ X\]
is also an admissible monic.

Since the first two horizontal rows and the middle column are short exact, then by the Noether Isomorphism for exact categories \cite[Lemma 3.5]{Bu} the third columnn is a well defined admissible short exact sequence, and is uniquely determined by the requirement that it makes the diagram commutative. Moreover, the upper right hand square is bicartesian;

\[ \xymatrix{ & & 0 \ar[d]& {\color{blue}0\ar[d]}\\ 0\ar[r] & X\ar@{=}[d] \ar@{>->}[r] & Y'{+}_Z X \ar@{>->}[d] \ar@{>>}[r] & {\color{blue}(Y'{+}_Z X)/X}\ar@{>->}[d] \ar[r]  \ar[d] & 0\\ 0 \ar[r] & X \ar@{>->}[r]  & Y''{+}_Z X \ar@{>>}[r] \ar@{>>}[d] & {\color{blue}(Y''{+}_Z X)/ X} \ar[r] \ar@{>>}[d] & 0 \\  &  & (Y''{+}_Z X)/(Y'{+}_Z X)\ar[r]\ar[d] & {\color{blue}(Y''{+}_Z X)/ (Y'{+}_Z X)} \ar[d]& \\ & & 0& {\color{blue}0} }\]

In particular $(Y''{+}_Z X)/(Y'{+}_Z X)$ is the admissible Cokernel of the admissible monic $c$.
\end{proof}

\begin{lemma}({\bf{The $\E-$second isomorphism theorem}})\label{parallelo}
Let $X$, and 
$Y'\mono Y''$ be three $\E-$subobjects of an object $Z$. 
The following is an admissible short exact sequence
\[Y'{\cap}_{Z}X\mono Y'\epi (Y'{+}_Z X)/X\]
\end{lemma}
\begin{proof}
We consider the following push-out diagram
\[\xymatrix{
Y'{\cap}_{Z}X \; \ar@{>->}[d]_{f} 
\ar@{ >->}[r]^{g} \ar@{}[dr]|{\text{PO}}  & Y' \ar@{>->}[d]^{f'}\\
X \; \ar@{>->}[r]^{g'} & Y'{+}_Z X }
\]
and by \cite[Proposition 2.12]{Bu} this square is part of the diagram
\[\xymatrix{
Y'{\cap}_{Z}X \; \ar@{>->}[d]_{f} 
\ar@{ >->}[r]^{g}  \ar@{}[dr]|{\text{PO}} 
& Y'\ar@{>->}[d]^{f'} \ar@{>>}[r]^{c} & Y'/(Y'{\cap}_{Z}X)\ar@{=}[d]\\
X \; \ar@{>->}[r]^{g'} &Y'{+}_Z X \ar@{>>}[r]^{c'} & (Y'{+}_Z X)/X.}
\]

\end{proof}

\begin{proposition}\label{the s.e.s}Let $X$, and $Y'\mono Y''$ be three $\E-$subobjects of an object $Z$. There exists an admissible short exact sequence \[ (Y''{\cap}_{Z} L)/(Y'{\cap}_{Z} L)\mono (Y''/Y') \epi (Y''{+}_{Z}X)/(Y'{+}_{Z}X).\]
\end{proposition}
\begin{proof}Consider the commutative diagram below in which the three colomus are admissibles short exact sequences by \ref{c3} and \ref{quotient}. In addition the first two rows are admissibles short exact sequences by \ref{parallelo}, then the 3$\times$3-lemma for exact categories \cite[Corollary 3.6]{Bu} implies the existence of the commutative diagram of admissible short exact sequences
\[ \xymatrix{0\ar[r] & Y'{\cap}_Z X \ar@{>->}[d]_{} \ar@{>->}[r]{}^{ } & Y'\ar@{>->}[d]_{} \ar@{>>}[r] &  (Y'+X)/X\ar@{>->}[d]^{} \ar[r] & 0\\ 0\ar[r] &Y''{\cap}_Z X \ar@{>>}[d]\ar@{>->}[r]_{} & Y'' \ar@{>>}[d] \ar@{>>}[r] & (Y''+_Z X)/X \ar@{>>}[d] \ar[r] & 0 \\ {\color{blue}0}\ar[r] & {\color{blue}(Y'{\cap}_Z X)/(Y'{\cap}_Z X)} \ar@{>->}[r]_{}  & {\color{blue}Y''/Y'} \ar@{>>}[r] & {\color{blue}(Y''+_Z X)/(Y'+_Z X) }\ar[r] & {\color{blue}0}} \] 
and in particular the third row is an admissible short exact sequence.
\end{proof}

\section{THE JORDAN-H\"OLDER PROPERTY}
In \cite{Bau}, Baumslag  gives a short proof of the Jordan–Hölder theorem, for \emph{groups}, by intersecting the terms in one subnormal series with those in the other series.\\ 
In this section we write Baumslag proof of the Jordan-H\"older theorem for abelian categories in the language of exact category $(\A, \E)$.\\
We repeat \cite{Bau} steps by using the admissible morphisms of the maximal exact structure of the abelian category. Our proof use only exact category theoretic arguments, in particular the Schur lemma for exact categories \ref{schur}.

\begin{definition}\label{composition series}
An $\mathcal{E}-$composition series for an object $X$ of $\mathcal{A}$ is a sequence 
\begin{eqnarray}\label{chain1} 0=X_0 \; \imono{i_0} X_1 \;\imono{i_1}  \cdots \; \imono{i_{n-2}} \;  X_{n-1}\;\imono{i_{n-1}}\; X_n=X \end{eqnarray}
where all $i_l$ are \emph{proper admissible monics} with $\E-$simple cokernel. 
\end{definition}

\begin{theorem} {\bf (Jordan-H\"older theorem)}\label{JH} Let $(\mathcal{A}, \mathcal{E})$ be an AIS-category. Any two $\mathcal{E}-$composition series for a finite object $X$ of $\mathcal{A}$
\[ 0=X_0 \; \imono{i_0} X_1 \;\imono{i_1}  \cdots \; \imono{i_{m-2}} \;  X_{m-1}\;\imono{i_{m-1}}\; X_m=X \] and \[0=X'_0 \; \imono{i'_0} X'_1 \;\imono{i'_1}  \cdots \; \imono{i'_{n-2}} \;  X'_{n-1}\;\imono{i'
_{n-1}}\; X'_n=X \]

are equivalent, that is, they have the same length and the same composition factors, up to permutation and isomorphism.  
 
\end{theorem}
\begin{proof}
By induction on $m$. If $m=0$, then $X=0$ and $n=0$.
If $m=1$, then $M$ is $\E-$simple: the only $\E-$composition series is $0\mono M$, and so $n=1$.
If $m\gneq 1$, we consider the sequence on $\E-$subobjects of $X$:
\[ 0 \; \imono{} X'_1{\cap}_X X_{m-1} \;\imono{}  \cdots \; \imono{} \;  X'_{n-1}{\cap}_X X_{m-1}\;\imono{}\;X_{m-1}=\] \[ X_{m-1}\; \imono{} X'_1 {+}_X X_{m-1}\;\imono{}  \cdots \; \imono{} \;  X'_{n-1} {+}_X X_{m-1}\;\imono{}\;X.\]
Since the Cokernels $X/X_{m-1}=X_{m}/X_{m-1}$ are $\E-$simples, there exists a unique $0\leqslant k\lneq n$ such that 
\[ X_{m-1}= X'_1 {+}_X X_{m-1}= \cdots X'_k{+}_{X} X_{m-1} {\subsetneq}_{\mathcal{E}} X'_{k+1} {+}_{X} X_{m-1}\cdots  =X'_{n-1} {+}_{X} X_{m-1}=X.\]
By 
\ref{the s.e.s}, there exists for each $0 \leqslant l\lneq n$ an admissible short exact sequence 
\[0\rightarrow (X'_{l+1}{\cap}_{Z} X_{m-1})/(X'_{l}{\cap}_{X} X_{m-1})\mono (X'_{l+1}/X'_{k}) \]
\[\space \epi (X'_{l+1}{+}_{Z}X_{m-1})/(X'_{l}{+}_{X}X_{m-1})\rightarrow 0.\]
In particular the middle term of this sequence is an $\E-$simple object. By the $\E-$Schur lemma \ref{schur}, the admissible monic (respectively the admissible epic) of this sequence is either the zero morphism, or an isomorphism.
For $l=k$, we have 
\[X_{m}/X_{m-1}\backsimeq  (X'_{k+1}{+}_{X}X_{m-1})/(X'_{k}{+}_{X}X_{m-1}) \backsimeq  (X'_{k+1}/X'_{k})\]
and then by \ref{zero coker} we have
$X'_{k+1}{\cap}_Z X_{m-1}\backsimeq X'_k{\cap}_X X_{m-1}$.
While for $l\neq k$ we have 
\[(X'_{l+1}{\cap}_{Z} X_{m-1})/(X'_{l}{\cap}_{X} X_{m-1})\backsimeq (X'_{l+1}/X'_{k})\]
which means that $X'_{l+1}{\cap}_X X_{m-1} \neq X'_l{\cap}_X X_{m-1}$ and $X'_{l+1}{\cap}_X X_{m-1} / X'_l{\cap}_X X_{m-1}$ is an $\E-$simple object. This shows that the sequence 
\[ 0 \; {\subsetneq}_{\mathcal{E}} X'_1{\cap}_X X_{m-1} {\subsetneq}_{\mathcal{E}} \cdots X'_k{\cap}_X X_{m-1}=X'_{k+1}{\cap}_X X_{m-1}\cdots {\subsetneq}_{\mathcal{E}} \;  X'_{n-1}{\cap}_X X_{m-1}{\subsetneq}_{\mathcal{E}} X_{m-1}\]
is a composition series of $X_{m-1}$ of length $n-1$. By the recurrence hypothesis
$m-1=n-1$, and so $m=n$ and there exists a bijection \[\sigma: \{0, 1,..., k-1, k+1, ..., n-1\}\rightarrow \{0, 1, ..., m-1\}\]
such that $X'_{l+1}/X'_l \backsimeq X_{\sigma(k)+1}/X_{\sigma(k)}$ for $l\neq k$, and by taking $\sigma(i)=m-1$.
\end{proof}

\begin{remark}
More generally, for a fixed additive category $\A$, one may choose an exact structure $\E$ on $\A$ from the lattice of exact structures $Ex(\A)$ (introduced in \cite[Section 5]{BHLR} and recentely studied in \cite{FG} and \cite{BBH}) and consider the $\E-$Jordan-H\"older property. Then the exact category $(\A, \E)$ may not necessarly satisfy the $\E-$Jordan-H\"older property (see \cite[Example 6.9]{BHLR}, \cite{E19} and \cite[Examples 5.3, 5.12]{BHT} for counter-examples) and characterisations of Jordan-H\"older exact categories has appeared in both \cite{E19} and in \cite{BHT}.
\end{remark}



\bigskip

{D\'{e}partment de math\'{e}matiques \\ 
 Universit\'{e} de Sherbrooke \\ 
 Sherbrooke, Qu\'{e}bec, J1K 2R1 \\ 
 Canada\\
Souheila.Hassoun@usherbrooke.ca\\
Sunny.Roy@usherbrooke.ca }


\begin{thebibliography}{99}




\bibitem[Ba06]{Bau}Baumslag, Benjamin (2006), "A simple way of proving the Jordan-Hölder-Schreier theorem", American Mathematical Monthly, 113 (10): 933–935.


\bibitem[Bi34]{Bi}Birkhoff, Garrett (1934), \emph{Transfinite subgroup series}, Bulletin of the American Mathematical Society, 40 (12): 847–850.
\bibitem[BBH]{BBH} R.-L.Baillargeon, Th.Br\"ustle, S.Hassoun, {\em On the Lattice of exact structures,} in preparation.

\bibitem[Br07]{Br} T. Bridgeland, {\em  Stability conditions on triangulated categories}, Ann. of Math. (2), 166(2):317--345, 2007. 


\bibitem[BHLR18]{BHLR} Th.Br\"ustle, S.Hassoun, D.Langford, S.Roy, {\em Reduction of exact structures,}
J. Pure Appl. Algebra 224 (2020), no. 4, 106212, 29 pp, arXiv:1809.0128.
\bibitem[BHT20]{BHT}Th.Br\"ustle, S.Hassoun, A.Tattar, \emph{Intersection, sum and Jordan-Holder property for exact categories},arXiv:.

\bibitem[BST]{BST} Th. Br\"ustle, D. Smith and H. Treffinger,  {\em Stability Conditions and Maximal Green Sequences in Abelian Categories}, arXiv: 1805.04382.


\bibitem[B\"u10]{Bu} T.Bühler, \emph{Exact categories.} Expo. Math. 28 (2010), no. 1, 1--69.








\bibitem[E19]{E19} H.Enomoto, \emph{The Jordan-H\"older property and Grothendieck monoids of exact categories,}
arXiv:1908.05446.

\bibitem[E20]{E20}H.Enomoto, \emph{Schur's lemma for exact categories implies abelian,} arXiv: 2002.09241.


\bibitem[FG20]{FG}X.Fang, M.Gorsky, \emph{Exact structures and degeneration of Hall algebras}, arXiv: 2005.12130, 2020.
\bibitem[G62]{Gabriel}P.Gabriel, \emph{Des cat\'egories ab\'eliennes} Bull. Soc. math. France, 90, 1962, p.323- 448.

\bibitem[GR92]{GR} P.Gabriel and A.V.Ro\u{\i}ter, \emph{Representations of Finite-dimensional Algebras,} in: Algebra, VIII, Encyclopaedia Mathematical Sciences, vol. 73, Springer, Berlin, 1992 (with a chapter by B. Keller), pp. 1--177. 
\bibitem[HSW20]{HSW} S.Hassoun, A.Shah, S-A.Wegner,
\emph{Examples and non-examples of integral categories }, arXiv:2005.11309, 2020.


\bibitem[Pa70]{par}B.Pareigis, \emph{Categories and functors}, University of Munich, Germany. Academic press, Newyork.London, 76-117631.

\bibitem[Po73]{Po} N.Popescu, \emph{Abelian categories with applications to rings and modules.} London Mathematical Society Monographs, No. 3. Academic Press, London-New York, 1973. xii+467 pp. 

\bibitem[Qu73]{Qu}
D.Quillen, \emph{Higher algebraic {$K$}-theory. {I}}, Algebraic
  $K$-theory, I: Higher $K$-theories (Proc. Conf., Battelle Memorial Inst.,
  Seattle, Wash., 1972), Springer, Berlin, 1973, pp.~85--147. Lecture Notes in
  Math., Vol. 341.

 
  

  
 
 
\bibitem[Rot95]{Rot}J.Rotman, \emph{An introduction to the theory of groups,} 4th ed, Springer-Verlag, 1995.
   
\bibitem[Ru97]{Ru} A. Rudakov. Stability for an abelian category. Journal of Algebra, 197:231–245, 1997.  
  
  






 
\end{thebibliography}
\end{document}